\documentclass[11pt]{article}
\usepackage{graphicx} 
\usepackage{svg}

\usepackage{tikz-cd}
\usepackage[utf8]{inputenc}
\usepackage{amsmath}
\usepackage{amsfonts}
\usepackage{amssymb}
\usepackage{amsthm}
\usepackage{booktabs}
\usepackage{latexsym}
\usepackage{ytableau}
\usepackage{tikz}
\usepackage[
    backend=biber,
    style=numeric,
  ]{biblatex}
\usepackage{float}
\usepackage{mathdots}
\usepackage[multiple]{footmisc}

\newtheorem{thm}{Theorem}
\newtheorem{prop}{Proposition}

\newtheorem{defin}{Definition}
\newtheorem{rmk}{Remark}

\newtheorem{cor}{Corollary}
\newtheorem{lem}{Lemma}

\newtheorem{notn}{Notation}

\DeclareMathOperator{\CC}{\mathbb{C}}
\DeclareMathOperator{\NN}{\mathbb{N}}
\DeclareMathOperator{\ZZ}{\mathbb{Z}}
\DeclareMathOperator{\QQ}{\mathbb{Q}}
\DeclareMathOperator{\HH}{\mathbb{H}}
\DeclareMathOperator{\GL}{\mathbf{GL}}

\DeclareMathOperator{\Sym}{Sym}
\DeclareRobustCommand{\binomer}{\genfrac\langle\rangle{0pt}{}}

\title{Quasiautomorphic forms are isomorphic to vector-valued automorphic forms}
    
\author{Michael Andrew Henry\footnote{The author acknowledges the assistance of Cameron Franc (McMaster University) with the representation theoretic arguments. We also thank Felipe Espreafico Guelerman (Sorbonne Universit\'{e}) for many conversations and clarifications about the exciting enumerative results from mathematical physics and enumerative geometry.}}

\date{\today}

\begin{document}

\maketitle

\begin{abstract}
    We utilize the structure of quasiautomorphic forms over a Hecke triangle group to define a mapping from a quasiautomorphic form to a vector-valued automorphic form (vvaf). This kind of vvaf we call a Hecke vector-form. First we supply a proof of the functional equations that hold for Hecke vector-forms modulo the group generators. Then, utilizing the multiplier system for these Hecke vector-forms, we prove the opposite direction and complete the bijection. Since the modular group is a special instance of the Hecke triangle groups, our results hold for quasimodular forms.  
\end{abstract}


\section{Introduction}
A \textbf{modular form} $ f(z) $ is a holomorphic function defined on the upper half-plane, $ \mathbb{H}  := \{z\in \mathbb{C}: \Im(z)>0 \} $, such that it obeys a simple functional equation on elements $ \frac{az + b}{cz+d} \in \HH $ or what is $ z \in \mathbb{H}$ under the action of the \textbf{full modular group}, $ \mathbf{SL}_2(\mathbb{Z}) $. These functions with variations on the underlying matrix group have been studied as far back as Riemann and Poincaré, where they were understood to be closely linked to the elliptic functions of Eisenstein, Jacobi, and Weierstrass \cite{Roy2017}. The interest in these functions has only increased over time, where they see diverse applications across mathematics, but especially in number theory, combinatorics, and mathematical physics. A lot of the interest in modular forms has come by studying the arithmetic of the coefficients of series expansions of particular modular forms — some of the astonishing work of Ramanujan is a prime example — where the integrality and combinatorial richness of those coefficients has proven decisive.

\textbf{\emph{Quasi}modular forms}, on the other hand, had their definition set out only very recently in the work \cite{kanekoZagier1995}. The prototypical example of a quasimodular form is the weight $ w = 2 $ Eisenstein series $ E_2 $, which fails to transform as a modular form (see \eqref{eq:E2FunctionalEqs}), but, when taken together with Eisenstein series $ E_4 $ and $ E_6 $, generates a ring closed under the differential operator $ \vartheta := z\frac{d}{dz} $ (a result due to Ramanujan). The aim of \cite{kanekoZagier1995} was to put theory behind and generalize a function appearing in the famous work of Dijkgraaf \cite{dijkgraaf1995}. It is less well-known, but in \cite{Douglas1995} the function can be found in a different form (as explained in \cite{LandoZvonkin2004}, p. 425). Mysteriously, this function shared many characteristics of a modular form, yet it deviated just slightly. For concreteness, the \textbf{Dijkgraaf-Douglas function} is given as a Fourier series (in $ q := e^{2 \pi i z} $) by
    \begin{eqnarray*}
        D_{6, 3}(z) &=& q + 8q^2 + 30q^3 + 80q^4 + 180q^5 +  336q^6 + 620q^7 + \cdots \\
        &=& \frac{1}{6} \sum_{n=1}^{\infty} (n\sigma_3(n) -n^2 \sigma_1(n))q^n,
    \end{eqnarray*}
where the latter identity is shown in \cite{kaminakaKato2021} if as usual $ \sigma_{m}(n) := \sum_{d|n} d^m $ (also see A$126858$; the paper \cite{mazur2004} together with the appendix of \cite{LandoZvonkin2004} gives a captivating analysis of the role played by enumerative topology). By \cite{kanekoZagier1995}, $ D_{6,3} $ is a weight $ w = 6 $ and depth $ r = 3 $ quasimodular form where it has relevance in the theory of mirror symmetry (of mathematical physics) as a generating function, its coefficients enumerating the generically ramified coverings of a genus $ g = 1 $ Riemann surface by Riemann surfaces of genus $ g = n  $ (\cite{LandoZvonkin2004}, 425). 

In regards to its appearance as a generating function, $ D_{6,3} $ was called extremal in \cite{kanekoKoike2006} because the series can be put in the form $ q^{d-1} + \sum_{k=0} a_k q^{d+k} $ where $ d $ is the dimension of the  vector-space where the function resides. This feature turned out to be a highly tractable notion and since then all extremal quasimodular forms with integral coefficients have been determined; it is known that none exist when the depth $ r $ exceeds $ 3 $ \cite{kaminakaKato2021}, \cite{nakaya2024}, making $ D_{6,3}$ even more remarkable. The paper \cite{nakaya2024} builds past work and gives many explicit formulas for coefficients of extremal quasimodular forms and also congruence relations in the spirit of Ramanujan's work on integer partitions \cite{Berndt2006}. There has also been progress on determining extremal quasiautomorphic forms associated to subgroups of $ \mathbf{SL}_2(\mathbb{Z}) $ which have integer coefficients, but this work is still ongoing (see discussion with references in \cite{nakaya2024}). 

To motivate the main theorem of our paper, recall that a \textbf{Hecke triangle group} $ \mathfrak{t}_\mu $ is generated in $ \mathbb{H} $ by two maps (or generators)
    \begin{eqnarray}\label{eq:HeckeGenerators}
          T_{\mathfrak{t}_\mu}= T := 
            \begin{pmatrix} 
                1 & \varpi_\mu \\ 0 & 1 
            \end{pmatrix}, \qquad
        S_{\mathfrak{t}_\mu} =S := 
            \begin{pmatrix}
                0 & -1 \\ 1 & 0 
            \end{pmatrix}, \qquad
    \end{eqnarray}
(when $ \varpi_{\mu} := 2 \cos \frac{\pi}{\mu} $) and the operation is matrix multiplication (\cite{beardon1983}). The group presentation with respect to $T$ and $S$ is written
    \begin{eqnarray*}
        \mathfrak{t}_\mu = \langle T, S : S^2 = (TS)^\mu = 1 \rangle,
    \end{eqnarray*}
and from this it is clear that the full modular group is just the Hecke triangle group $ \mathfrak{t}_3 = (2,3,\infty) $. It is known that for each Hecke triangle group $\mathfrak{t}_\mu $, in exact analogy with the modular case, we may define functions called automorphic forms that satisfy
    \begin{eqnarray*}
        {g\left(\frac{az +b}{cz+d}\right)} = {(cz+d)^{2k}} g(z), \quad \text{whenever } \begin{pmatrix}
            a & b\\
            c & d
        \end{pmatrix}
        \in \mathfrak{t}_\mu 
    \end{eqnarray*}
for  $ 2 \leq k \in \mathbb{N} $ \cite{doranEtAl2013}, \cite{lehner1964}. Yet, somehow these automorphic functions have enjoyed nothing like the attention given to modular forms. We suspect that a reason for this neglect is that, if $ g $ is an automorphic function associated to $ \mathfrak{t}_\mu $ when $ \mu \not \in \{3, 4, 6\} $ (that is, $ \mathfrak{t}_\mu $ is not arithmetic \cite{Takeuchi1977}), then it was proven in \cite{wolfart1983} in the early 1980's that the coefficients in the series development of $ g $ are irrational numbers. Thus, much of the interest and success of modular forms, which has come in so many cases by studying the coefficients, cannot be carried over to irrational numbers. 

The situation is past due for a reappraisal, however, since recent mathematical developments are legitimizing non-integer counting functions. Specifically, in some very challenging areas of mathematics, such as string theory and enumerative geometry, the landscape of enumerative combinatorics is being developed in unexpected diretions: more and more studies are finding ways of attaching discrete combinatorial meaning to rational numbers. The most famous examples are probably the so-called Gromov-Witten invariants. Originating in string theory, these numerically rational invariants, while hard to define simply, are often calculated by determining coefficients in infinite series, just as with so many examples from classical combinatorics \cite{wilf1994}. The reader may consult \cite{rose2014} and \cite{yui2011} for examples of generating functions with Gromov-Witten invariants and can see ch. 4 in \cite{KockVainsencher2007} for a definition. Even more dramatically though, the recent work \cite{walcher2012} in enumerative geometry assigns discrete combinatorial meaning to irrational numbers that appear in polylogarithm expansions. The techniques appear to the author (who is not an expert) to descend from the counting revolution that had already begun taking place in physics. These advances compel us to return to automorphic forms on other groups such as the non-arithmetic $ \mathfrak{t}_\mu $ as potentially rich sources of combinatorial information that have been overlooked. The present article is progress in this direction.

Having provided some context, we arrive at the main question motivating this article: what exactly is a quasiautomorphic form? We will answer this completely by constructing an isomorphism between quasiautomorphic forms and some much better understood objects. Precisely, for $ U $ a quasiautomorphic form, we determine a bijective mapping $ U\, \leftrightarrow\, \mathbf{F}_U$ such that $ \mathbf{F}_U(z) $ is a vector-valued automorphic form. In our terminology, this function $\mathbf{F}_U(z) $ is a called a \textbf{Hecke vector-form}. Now we quote the main result of the article:
    \begin{thm}[Isomorphism theorem]
        The function $ U_{\mathfrak{t}_\mu, w, r}(z) = U $ is a quasiautomorphic form (with respect to $ \mathfrak{t}_\mu $) of weight $ w = 2k $ and depth $ 0 \leq r \leq \frac{w}{2}  $ if and only if there is a vector-valued automorphic form $ \mathbf{F}_U(z) $ such that
            \begin{eqnarray*}
                    \mathbf{F}_U(Tz) &=& \varepsilon_{r}(T)\,\mathbf{F}_U(z), \\
                    \frac{\mathbf{F}_U(Sz)}{z^{w-r}} &=& \varepsilon_r(S)\, \mathbf{F}_U(z)  
            \end{eqnarray*}
        for generators $ T_{\mathfrak{t}_\mu} = T $ and $  S_{\mathfrak{t}_\mu } =  S $ of $ \mathfrak{t}_\mu $, where the multiplier system $\varepsilon_r $ is given by natural, simple, and integral $ (r+1) \times (r+1) $ matrices.
    \end{thm}
It turns out that our multiplier system $ \varepsilon_r $ is generic in that it is independent of $ U $; only $ \varepsilon_r(T) $ takes the parameter value $ \varpi_\mu = 2 \cos \frac{\pi}{\mu} $ which appears in the generator $ T $.    

Our Theorem now allows quasiautomorphic forms to fall under the general methods of Fuchsian $n$-th order linear differential equations whose solutions are invariant with respect to $ T $ and $ S $. Conforming closely to existing convention, these might be called \textbf{automorphic linear differential equations} (or ALDEs), solutions of which are \textbf{vector-valued automorphic forms} (or vvafs) (we prefer the term vector-form). Our Theorem is then a natural and substantial strengthening of the existing, more familiar theory of \textbf{modular linear differential equations} (or MLDEs), solutions being \textbf{vector-valued modular forms} (or vvmfs). What is promising in our approach is that MLDEs already have a very large literature and thus a body of facts and techniques that predates quasiautomorphic forms considerably \cite{francMason2016}, including already known connections to vertex operator algebras and conformal field theory (CFT). All of this is now available, via our bijection, as a toolkit for the study of quasiautomorphic forms.

\section{Setup and notation}
In this section we review the essential background concepts and set most of the notation for the remainder of the article.

\subsection{Hecke triangle groups}

To begin, let us recall that if $ p, q, r \in (\NN \setminus \{1\}) \cup \{\infty\}$ satisfy
    \begin{eqnarray}\label{eq;firstKindCriterium}
        \frac{1}{p} + \frac{1}{q} + \frac{1}{r} < 1,        
    \end{eqnarray}
then the \textbf{Schwarz triangle group} $ (p,q,r) $ is the group with presentation
    \begin{eqnarray*}
        (p,q,r) := \langle\, x, y, z \;\mid\; x^p = y^q = z^r = xyz = 1\,\rangle;
    \end{eqnarray*}
in the case, say, $ r = \infty $, then we define $ 1/r := 0 $ \cite{beardon1983}. The Schwarz triangle group leads us directly to the group we treat here in that it is a subgroup:
    \begin{defin}[Hecke triangle group]
        For an integer $ \mu\geq 3$, the \textbf{Hecke triangle group} is the special Schwarz triangle group
            \begin{eqnarray*}
                \mathfrak{t}_\mu := (2, \mu, \infty).        
            \end{eqnarray*}
    \end{defin}
Recall that the generators of $\mathfrak{t}_\mu $ have already been shown at \eqref{eq:HeckeGenerators}. To connect with the triplet notation, one hyperbolic triangle of group $\mathfrak{t}_\mu $ has angles $  \frac{\pi}{2}, \frac{\pi}{\mu}, 0 $, where the last vertex is a cusp (see Figure \ref{fig:GlobCurveTriangle}). 

 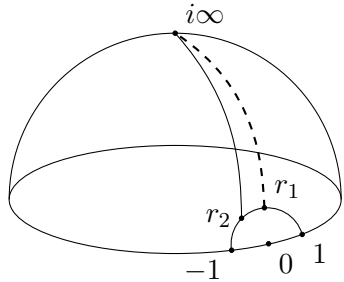
\begin{figure}[H]
        \centering
        \begin{tikzpicture}[scale=.55]
 
            \draw[] (-4,0) arc (180:0:4cm and 4cm);
 
            \draw[] (0,0) ellipse (4cm and 1.3cm);
 
            \draw[] (2.2194,-1.0720) ++ (15.4:0.8720) arc (15.4:189.8:0.8720);
 
            \draw[thick, dashed] (2.15,-0.2000)
            .. controls (2.10, 0.84) and (2, 2.7) .. 
            (0.00, 4.00);
 
            \draw[] (1.6,-0.42)
            .. controls (1.58, 0.765) and (1.5, 2.325) ..
            (0.00, 4.00);
 
            \filldraw (2.25,-1.065) circle (1.5pt)
            node[below right, font=\itshape] {$0$};
 
            \filldraw (0.00, 4.00) circle (1.5pt)
            node[above right] {$i \infty$};
 
            \filldraw (1.3601,-1.2204) circle (1.5pt)
            node[below left, font=\itshape] {$-1$};
 
            \filldraw (3.0601,-0.8404) circle (1.5pt)
            node[below right, font=\itshape] {$1$};
 
            \filldraw (2.15,-0.2000) circle (1.5pt)
            node[above right, font=\itshape] {$r_1$};
 
            \filldraw (1.6,-0.45) circle (1.5pt)
            node[left, font=\itshape] {$ {r_2} $};
        \end{tikzpicture}
        \caption{Global depiction of normalized initial triangle.}
        \label{fig:GlobCurveTriangle}
    \end{figure}

A triangle together with its reflection is the \textbf{fundamental domain} for $ (2,\mu, \infty) $ (see Figure \ref{fig:LocCurveTriangle}). Apropos later considerations, it is known that a Riemann mapping function from such a hyperbolic triangle to $\HH $ is single valued (\cite{nehari1952}, p. 310).

We shall want a convenient normalization of $\mathfrak{t}_\mu$ for use with automorphic forms (and related quasiautomorphic forms) and we briefly explain this.
  
  \begin{figure}[H]
        \centering
        \begin{tikzpicture}[z=-1 cm, scale = 1. 45]
            \coordinate [label=below:\textcolor{black}{$r_1$}]  (A) at (0,1);
            \coordinate [label=below:\textcolor{black}{$r_2 $}]  (B) at (-.5,0.866);
            \coordinate [label=below:\textcolor{black}{$0$}]  (D) at (0,0);
            \coordinate [label=below:\textcolor{black}{$-1$}]  (D) at (-1,0);
            \coordinate [label=below:\textcolor{black}{$1$}]  (D) at (1,0);
            \coordinate [label=above:\textcolor{black}{$ i\infty $}]  (G) at (0, 2.4);
            \coordinate [label=right:\textcolor{black}{$ d $}]  (G) at (-0.42, 1.7);
            \draw[] (-0.5,.866) -- (-0.5,2.3);
            \draw[color=black, dashed, thin] (0,1) -- (0,2.3) ;
            \draw[] (0.5, 0.866) -- (0.5,2.3);
            \draw[] (-1.5,0,0) -- (1.5,0,0);
            \draw[] (-1,0) .. controls (-1,0.555) and (-0.555,1) .. (0,1)
            .. controls (0.555,1) and (1,0.555) .. (1,0);
        \end{tikzpicture}
        \caption{Locally, the normalized initial triangle $ d $ with vertices $r_1$ and $r_2 $, etc.}\label{fig:LocCurveTriangle}
    \end{figure}
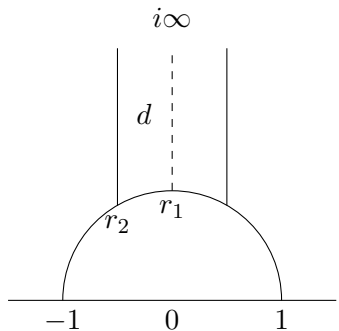

Let the corner vertices $ r_1, r_2, $ and $ r_3 $ of the initial curvilinear triangle be
    \begin{eqnarray}\label{normalizedVertices}
            r_1 = -{e^{-\pi i/2 }}, \qquad
            r_2 = -e^ {-\pi i/\mu}, \qquad
            r_3 = \infty;
    \end{eqnarray}
this is the standard normalization of $ \mathfrak{t}_\mu $ from number theory and we assume without further mention that a Hecke triangle group is normalized this way.

\subsection{Automorphic forms}

Following closely the presentation from \cite{zemel2015}, we let 
    $$ \gamma = 
        \begin{pmatrix}
            a & b\\
            c & d
        \end{pmatrix} \in \mathfrak{t}_\mu, 
    $$ 
then the action of $ \gamma $ on $ \HH $ we denote as a juxtaposition with $ z $ as rightmost, namely
    \begin{eqnarray*}
           \gamma z := \frac{az+b}{cz+d}.
    \end{eqnarray*}
With respect to this action, an artifact occurs in the definition of an automorphic form that is called a \textbf{factor of automorphy} and defined
    \begin{eqnarray*}
        j(\gamma, z) := cz+d.
    \end{eqnarray*}
This factor of automorphy is a \textbf{cocycle} since whenever $ z \in \mathbb{H} $ and $ \gamma, \delta \in \mathfrak{t}_\mu $, the identity
    \begin{eqnarray*}
        {j(\gamma \delta, z)} = j(\delta, z)\,{j(\gamma,\delta z )} 
    \end{eqnarray*}
holds. We now have everything to present

    \begin{defin}\label{def:automorphicForm}
        An \textbf{automorphic form} $ f $ of weight $ (w,u) \in \mathbb{Z} \times \mathbb{Z} $ and representation $ \varrho: \mathfrak{t}_\mu \rightarrow \mathbf{GL}(V_\varrho) $ (where $ V_\varrho$ is the representation space of $ \varrho $) is a function $f: \HH \rightarrow \CC $ satisfying
            \begin{eqnarray*}
                f\left(\gamma z\right) = j(\gamma, z)^w {\overline{j(\gamma, z)}}^u \varrho(\gamma)f(z) \quad \text{when } \gamma \in \mathfrak{t}_\mu.
            \end{eqnarray*} 
    \end{defin}

Definition \ref{def:automorphicForm} is much more general than what we need, but we have given it for the sake of completeness (note that we will not continue at this level of generality). We assume here-forward that $ u = 0 $ and therefore we consider automorphic forms of weight $ w $, not $ (w,u) $, without further mention. A common practice is to denote the weight of the automorphic form as a subscript as in $ f_w $, a practice we will use where it makes sense. Also, as a result of our normalization of $ \mathfrak{t}_\mu $ (recall Figure \ref{fig:LocCurveTriangle}), the automorphic forms that we consider here all have representation $ \varrho(\gamma) = 1 $ or have the trivial representation. It follows also then that $ w = 2k $ holds for $ 1 \leq k \in \mathbb{N} $.

We consider here only holomorphic automorphic forms, which includes when $ z \rightarrow i \infty $. Holomorphic automorphic forms belong to a finite dimensional vector space depending on their weight $w$ and we denote this space by $ \mathcal{A}_w(\mathfrak{t}_\mu) $. Exact dimension formulas are non-trivial and known generally for automorphic forms over $ \mathfrak{t}_3 $, $\mathfrak{t}_4 $, and $ \mathfrak{t}_6 $ (the well-studied arithmetic cases). Note that the precise dimension depends on the representation $ \varrho $. Following \cite{berndtKnopp2008}, and assuming the trivial representation $ \varrho(\gamma)= 1 $ when $\gamma \in \mathfrak{t}_\mu $, then for $ w \equiv 0 \,(\text{mod } 4) $ satisfied, the formula
    \begin{eqnarray*}
        \dim \mathcal{A}_{w}(\mathfrak{t}_\mu) = \left\lfloor\frac{w}{4}\left( \frac{\mu -2 }{\mu} \right) \right \rfloor + 1
    \end{eqnarray*}
holds. For a proof see \cite{berndtKnopp2008}, beginning p. 52.

Finally, adapted from \cite{francMason2016}, we give 
    \begin{defin}
        A \textbf{vector-valued automorphic form} $ {F} $ of weight $ w $, order $n \in \NN $, and representation $ \varrho: \mathfrak{t}_\mu \rightarrow \mathbf{GL}_n (\ZZ) $ is a function $ F: \HH \rightarrow \CC^n $ satisfying:
            \begin{enumerate}
                \item $F(z)$ is holomomorphic,
                \item $F\left(\gamma z\right) = j(\gamma, z)^w \varrho(\gamma) F(z) \quad \text{when } \gamma \in \mathfrak{t}_\mu $,
                \item $ F(z) $ has a Fourier series expansion.
            \end{enumerate}
    \end{defin}

We end the subsection with a note. The term ``modular form'' we will reserve here for an automorphic form defined precisely with respect to the full modular group, $ \mathfrak{t}_3 = \mathbf{SL}_2(\mathbb{Z}) $, a convention that is not always obeyed in the literature and furthermore is a common source of confusion. The precedent for our choice goes back to Poincar\'{e} \cite{lehner1964}. 

\subsection{Quasiautmorphic forms}

To describe quasiautomorphic forms we depart from the exposition \cite{zemel2015} to give some additional motivation. Ramanujan famously showed that under the differential operator $ \vartheta := z \frac{d}{dz} $, the differential system 
    \begin{eqnarray*}
            \vartheta E_2 &=& \frac{1}{12}(E_2 \,E_2- E_4 ), \\
            \vartheta E_4 &=& \frac{1}{3}(E_2 \, E_4 - E_6), \\
            \vartheta E_6 &=& \frac{1}{2}(E_2 \, E_6 - E_4^2)    
    \end{eqnarray*}
holds among the Eisenstein series $ E_2 $, $ E_4 $ and $ E_6 $. The functions $ E_4 $ and $ E_6 $ are known to be modular forms of weights $ 4 $ and $ 6 $, respectively; however, not made fully clear until \cite{kanekoZagier1995}, the function $ E_2 $ is a quasimodular form of weight $ w = 2 $ and depth $ r = 1 $. It exists on a vector space of dimension $ d = 1 $ i.e. is unique up to constant multiples. With respect to generators $ T $ and $ S $ of $ \mathfrak{t}_3 $, it is known that the functional equations
    \begin{eqnarray}\label{eq:E2UnderS}
        E_2(Tz) = E_2(z), \qquad
        \frac{E_2(Sz)}{z^2} = E_2(z) - C_{\mathfrak{t}_3 }Sz 
    \end{eqnarray}
hold where we set $ C_{\mathfrak{t}_3} := \frac{6}{\pi i} $ see . 

For arbitrary $ \mathfrak{t}_\mu $, the situation is nearly identical  \cite{doranEtAl2013}: we can consider an analogous weight $ w = 2 $ and depth $ r = 1 $ Eisenstein series with respect to $\mathfrak{t}_\mu$ and for this we naturally write $ E_{\mathfrak{t}_\mu, 2} $. This function it turns out satisfies
    \begin{eqnarray}\label{eq:quasiAutoUnderTAndS}
        E_{\mathfrak{t}_\mu, 2}(Tz) = E_{\mathfrak{t}_\mu, 2}(z), \qquad
        \frac{E_{\mathfrak{t}_\mu, 2}(Sz)}{z^2} = E_{\mathfrak{t}_\mu, 2}(z) - C_{\mathfrak{t}_\mu }Sz, 
    \end{eqnarray}
where the \textbf{structure constant} is $ C_{\mathfrak{t}_\mu} := \frac{\text{lcm}(2, \mu)}{\pi i} $. The general definition of quasiautomorphic form for arbitrary weight and depth utilizes this special Eisentein series $ E_{\mathfrak{t}_\mu, 2} $. Thus, in a sense $ E_{\mathfrak{t}_\mu, 2} $ is the most essential piece. 

Corresponding to Ramanujan's result, for the operator $ \vartheta := \frac{z}{\varpi_\mu}\frac{d}{dz} $, we get a differential system 
    \begin{eqnarray*}
        \vartheta E_{\mathfrak{t}_\mu, 2} &=& \frac{\mu-2}{4\mu}\, E_{\mathfrak{t}_\mu,2}E_{\mathfrak{t}_\mu, 2}  - \frac{\mu-2}{4\mu}E_{\mathfrak{t}_\mu,4}, \\
        \vartheta E_{\mathfrak{t}_\mu, 4} &=& 2\,\frac{\mu-2}{2\mu}E_{\mathfrak{t}_\mu,2}\, E_{\mathfrak{t}_\mu,4} - \frac{\mu-2}{\mu}E_{\mathfrak{t}_\mu,6} \\
        \vartheta E_{\mathfrak{t}_\mu,6} &=& 3\,\frac{\mu-2}{2\mu}E_{\mathfrak{t}_\mu,2}\, E_{\mathfrak{t}_\mu,6} -\frac{\mu-3}{\mu}E_{\mathfrak{t}_\mu,8} - \frac{3-2}{2}E_{\mathfrak{t}_\mu,4}\,E_{\mathfrak{t}_\mu,4} , \\
        \vartheta E_{\mathfrak{t}_\mu,8} &=& 4\,\frac{\mu-2}{2\mu}E_{\mathfrak{t}_\mu,2}\, E_{\mathfrak{t}_\mu,8} -\frac{\mu-4}{\mu}E_{\mathfrak{t}_\mu,10} - \frac{4 - 2}{2}E_{\mathfrak{t}_\mu,4}\, E_{\mathfrak{t}_\mu,6}, \\
        & \vdots & \\
        \vartheta E_{\mathfrak{t}_\mu, 2 \mu} &=& \mu\, \frac{\mu-2}{2\mu}E_{\mathfrak{t}_\mu,2} \,E_{\mathfrak{t}_\mu, 2\mu}-\frac{\mu - 2}{2}E_{\mathfrak{t}_\mu,4} \, E_{\mathfrak{t}_\mu, 2\mu-2}  
    \end{eqnarray*}
is satisfied for the analogous Eisenstein series $ E_{\mathfrak{t}_\mu, w} $; as we would expect from the modular case are automorphic forms of weight $ w = 2k$, etc. Proofs of the differential system can be found in  \cite{ashokEtAl2020I}, \cite{doranEtAl2013}, \cite{henry2025}. 

We can now give 
    \begin{defin}
        If $ 0 \leq r \leq w/2 $, then a \textbf{quasiautomorphic form} of weight $ w $, depth $ r $, and with respect to $\mathfrak{t}_\mu$  is defined
            \begin{equation}\label{eq:standarForm}
                U_{\mathfrak{t}_\mu, w,r}(z) :=  \sum_{k=0}^{r} H_{\mathfrak{t}_\mu, w-2k}(z)\, E^k_{\mathfrak{t}_\mu, 2 } (z) 
            \end{equation}
        where $ H_{\mathfrak{t}_\mu, w-2k} \in \mathcal{A}_{w-2k}(\mathfrak{t}_\mu) $ is satisfied and $ E_{\mathfrak{t}_\mu, 2} $ is the aforementioned unique, weight $w =2 $ and depth $r=1$ Eisenstein series.
    \end{defin}
By the finite dimensionality of $\mathcal{A}_w(\mathfrak{t}_\mu) $, it follows that quasiautomorphic forms of weight $ w $ and depth $ r $ exist on a finite dimensional space that we will denote $ \mathcal{QA}_{w,r}(\mathfrak{t}_\mu) $. A dimension formula for $\mathcal{QA}_{w,r}(\mathfrak{t}_\mu) $ follows from $ \dim \mathcal{A}_{w}(\mathfrak{t}_\mu) $ after slightly tedious calculations; we do not give such a formula, nor do we need it in this paper. Regarding the notation $ U_{\mathfrak{t}_\mu, w, r} $, we will oscillate between writing $ U $, $ U_{ w, r} $, $U_{\mathfrak{t}_\mu, w, r} $, depending on the desired emphasis.

\section{Hecke vector-forms}

In this section, we use what was introduced prior to eventually describe precisely a Hecke vector-form $\mathbf{F}_U(z) $ (dependent on some quasiautomorphic form $U$). If $U$ has depth $r$, then the function $ \mathbf{F}_U(z) $ is a vector-valued function from $\HH \rightarrow \CC^{r+1}$ ; however, we will first introduce the component functions in scalar terms. In the following subsections we will collect this and put it into language of linear algebra.

\subsection{Binomial coefficients and related structures}

To define our vector function $ \mathbf{F}_U(x) $, we will need to make use of the special properties of binomial coefficients. A reader should be aware that our treatment of binomial coefficients is slightly idiosyncratic in order to better meet the needs of our exposition. 

In \cite{comtet1974} they introduce in a pair the functions
    \begin{eqnarray*}
        \binom{z}{k} := \frac{z(z-1)\cdots(z-k+1)}{k!}, \qquad
        \binomer{z}{k} := \frac{z(z+1)\cdots(z+k-1)}{k!}.  
    \end{eqnarray*}
The first is the usual binomial coefficient, while the second satisfies
    \begin{eqnarray*}
        \binomer{z+1}{k} = \binom{z+k}{k}.
    \end{eqnarray*}
It is fitting to use the contrasting notation here. We also introduce and use throughout our notation
    \begin{eqnarray*}
        \{ r,\ell \}_m :=  \frac{(r-\ell)!}{r!} \frac{(m+\ell)!}{m!}  
    \end{eqnarray*}
to avoid clutter (for us the condition $ r \geq \ell \geq 0 $ will always be satisfied, $ \ell, r \in \mathbb{N}\cup\{0\} $ holding). Observe that the equality 
    \begin{eqnarray}\label{eq:coeffEquiv}
        \binomer{m+1}{p}\{r,\ell \}_{m+1} = \binom{r-\ell}{m}\{r,\ell + m\}_p
    \end{eqnarray}
holds.

For motivation, beginning with $ U_{\mathfrak{t}_\mu, w,r} $, we will consider related functions $ f_i $ and $ g_i $ for $ 0\leq i \leq r $ which are are dependent on $ U_{\mathfrak{t}_\mu, w,r} $. What is important is they satisfy the relations
    \begin{eqnarray*}
        f_0(z) &=& g_0 \\
        f_1(z) &=& g_0z + g_1  \\
        f_2(z) &=& g_0z^2 + 2g_1z + g_2  \\
        f_3(z) &=& g_0z^3 + 3g_1z^2 + 3g_2z + g_3  \\
        & \vdots & 
    \end{eqnarray*}
Thus, in general,
    \begin{eqnarray}\label{eq:defFAux}
        f_n(z) := \sum_{k=0}^n \binom{n}{k}g_{k}z^{n-k},
    \end{eqnarray}
holds assuming $ 0 \leq n \leq r\leq w/2 $ and $ g_{\ell} $ is the function
    \begin{eqnarray}\label{ex:defGAux}
        g_{\ell}(z) := C^\ell  \sum_{m=0}^{r-\ell}\{ r,\ell \}_m B_{U, \ell + m}(z)\, E_{2}^{m}(z),
    \end{eqnarray}
if $ C_{\mathfrak{t}_\mu} = C  $ is the again unique structure constant associated to the weight $ w =2 $ Eisenstein series $ E_{2}(z) $ modulo $ \mathfrak{t}_\mu $ (i.e. $ \lim_{z\rightarrow i\infty}E_{\mathfrak{t}_\mu, 2}(z) = C_{\mathfrak{t}_\mu} $
holds), and $  B_{U, k} : = H_{w-2k}  $ from the sum definition of $ U_{\mathfrak{t}_\mu, w,r} $ given at \eqref{eq:standarForm}. For a stripped down example of these polynomials see Table \ref{tab:interimPolynomials}. As we should expect from \eqref{ex:defGAux}, $ z $ will always be restricted to the upper-half complex plane. 

  \begin{table}[H]
        \centering
        {
        \def\arraystretch{1.20}
        \begin{tabular}{|c|c|}
             \hline
             $\ell $ & $ p_{r-\ell}(z) = \sum_{m=0}^{r-\ell} \{r,\ell\}_mz^m $  \\
             \hline
             $ 0 $ & $ z^7 + z^6 + z^5 + z^4 + z^3 + z^2 + z + 1 $   \\
             $ 1 $ &  $\frac{1}{7}\left( 7z^6 + 6z^5 +  5 z^4 + 4z^3 + 3z^2 + 2z + 1\right) $   \\
             $ 2 $ &   $\frac{1}{21}\left(21 z^5 + 15 z^4 + 10z^3+ 6z^2 +3z +1\right) $ \\
             $ 3 $ &   $\frac{1}{35}(35 z^4 + 20z^3 + 10z^2 + 4z +1) $ \\
             $ 4 $ &   $\frac{1}{35}(35 z^3 + 15 z^2 + 5z + 1) $ \\
             $ 5 $ &  $\frac{1}{21}(21z^2 + 6z + 1)$  \\
             $6 $ & $\frac{1}{7}(7z + 1)$ \\
             $7$ & $1$ \\ 
             \hline
        \end{tabular}
        }
        \caption{Example when $ r = 7 $ of the polynomials occurring in $ g_\ell $ (without automorphic weights). Note the palindromic leading coefficients when read vertically.}
        \label{tab:interimPolynomials}
    \end{table}

It is evident that $ g_0 $ is a weight $ w $ and depth $ r $ quasiautomorphic form modulo $ \mathfrak{t}_\mu $; similarly, each function $ g_k $ is a quasiautomorphic form of weight $ w-2k $, depth $ r-k  $ or what can be viewed simply as a truncation of $ g_0 $ such that homogeneity of the weight is preserved in each term.

We pause the development briefly to motivate our construction using classical combinatorics. The finite sequence $ \mathbf{f} = (f_k)^r_{k=0} $ of elements given at \eqref{eq:defFAux} comes from defining two auxiliary sequences $ \mathbf{g} = (g_\ell)^r_{\ell =0}
 $ and $ \mathbf{z} = (z^j)^r_{j=0} $ such that the equality
    \begin{eqnarray}\label{ex:fAux}
        f_n = \sum_{k=0}^{n}\binom{n}{k}g_{k}{z^{n-k}}    
    \end{eqnarray}
holds i.e. $ f_n $ is defined as the $ n $-th \textbf{binomial convolution} of $ \mathbf{g} $ and $ \mathbf{z} $ where $ 0 \leq n \leq r $ holds. A convolution may be found written elsewhere in our paper as $ f_n = (\mathbf{g}\oplus\mathbf{z})_n $, but we mostly avoid this notation. For more background on this idea see \cite{grahamKnuthPatashnik1989}. One immediate implication of the set-up is we have that the so-called orthogonality relation
    \begin{eqnarray*}\label{eq:orthogonality}
        \frac{f_n}{z^n} = \sum_{i=0}^{n} \binom{n}{i}\frac{g_i}{z^{i}} \quad \Leftrightarrow \quad \frac{g_n}{z^n}= \sum_{j=0}^{n}(-1)^{n-j} \binom{n}{j}{f_j}{z^j} 
    \end{eqnarray*}
holds.

\subsection{Matrix theory}
We will need some special matrices, some for which we use non-standard notation (the reader can find any unmentioned background in \cite{hornJohnson1985}). 
    \begin{rmk}
        In a matrix, the absence of any indicating entry $ a_{i,j} $ always means $ a_{i,j} = 0 $ i.e. a blank space is to be assumed to take a zero.
    \end{rmk}
    
We can motivate the introduction of matrix theory here by observing
    \begin{eqnarray}\label{eq:matrixConvolutionIdentity}
        \begin{pmatrix}
            f_0 \\ f_1 \\ \vdots \\ f_r
        \end{pmatrix}
            =
        \begin{pmatrix}
            \binom{0}{0} \\
            \binom{1}{0}z & \binom{1}{1} \\
            \vdots &\vdots & \ddots \\
            \binom{r}{0}z^r & \binom{r}{1}z^{r-1} & \cdots & \binom{r}{r} 
        \end{pmatrix}
        \begin{pmatrix}
            g_0 \\
            g_1 \\
            \vdots \\
            g_r
        \end{pmatrix}
    \end{eqnarray}
holds, involving some pleasingly natural and well-behaved matrices. Clearly, \eqref{eq:matrixConvolutionIdentity} is just the matrix summary of the (binomial) convolution structure we saw in the previous section.

    \begin{defin}
        For the $ (r+1)\times(r+1) $ lower triangular matrix known as the \textbf{generalized Pascal matrix}, we write
            \begin{eqnarray*}
                P_r(z) := 
                    \begin{pmatrix}
                        \binom{0}{0} \\
                        \binom{1}{0}z & \binom{1}{1} \\
                        \vdots &\vdots & \ddots \\
                        \binom{r}{0}z^r & \binom{r}{1}z^{r-1} & \cdots & \binom{r}{r} 
                    \end{pmatrix}.
                \end{eqnarray*}
            \end{defin}
For more about these matrices see e.g. \cite{acetoTriguante2001} and \cite{edelmanStrang2004} and the references cited therein.  
    \begin{defin}\label{def:creationMatrix}
        Let
            \begin{eqnarray*}
                A_r(z) := 
                    \begin{pmatrix}
                         z  \\
                        1 & z  \\
                        & \ddots & \ddots &  \\
                        & & r& z
                    \end{pmatrix}.
            \end{eqnarray*}
        Then the \bf creation matrix \normalfont is defined by the special instance,  
        \begin{eqnarray}
              A_r := A_r(0).
        \end{eqnarray}. 
    \end{defin}
    
The terminology in Definition \ref{def:creationMatrix} comes from the paper \cite{acetoMalonekTomaz2015}; their study indicates that our approach defining $ \mathbf{F}_U $ already has a precedent in theory of special functions. Note that it is both well-known and useful that the relation $A_r^{s} = {0}$ holds for $ s \geq r + 1 $ i.e. the creation matrix $ A_r $ is a nilpotent matrix of index $ r+1 $. Additionally, it can be checked that the characteristic polynomial of $ A_r^{\mathbf{t}}(\lambda) $ satisfies
    \begin{eqnarray*}
        \text{{char}}_{A_r(z)}(X) = \sum^{r}_{k=0}\binom{r}{k}X^{r-k}z^k = (X - z)^r.
    \end{eqnarray*}
One interpretation of this via the Cayley-Hamiliton theorem is $ A_r^{\mathbf{t}}(z) $ is similar to a Jordan matrix, $ J_r(z) $, one of the most fundamental matrices \cite{hornJohnson1985}.

    \begin{defin}
        The $(r+1) \times (r+1) $ \textbf{exchange matrix} we denote 
            \begin{eqnarray*}
                \iota_r := 
                    \begin{pmatrix}
                        &&&1\\
                        &&1\\
                        & \iddots \\
                        1
                    \end{pmatrix}.
            \end{eqnarray*}
    \end{defin}

The exchange matrix (as it is called in \cite{hornJohnson1985}) is often hidden behind the matrix transpose operator. For example, with some loss of generality, for a $ (r+1) \times (r+1) $ matrix $ M $, the relation $ M^\mathbf{t} = \iota_r M \iota_r  $ holds. We shall want to have on hand a left ``half-transpose'' and a right ``half-transpose,'' namely the dual operators
    \begin{eqnarray}\label{eq:halfTranspose}
        M^\mathbf{x} := \iota_r M \quad \text{and} \quad
        M^\mathbf{y} := M\iota_r.
    \end{eqnarray}

We shall also want to use the familiar $(r+1) \times (r+1)$ diagonal matrix that we denote by
    \begin{eqnarray}\label{eq:diagMatrix}
        d_r(a_j) :=
            \begin{pmatrix}
                a_0\\
                & a_1 \\
                &&\ddots \\
                &&& a_r
            \end{pmatrix}.
    \end{eqnarray}

\subsection{Simple consequences of the matrix approach}
For the first application of matrix theory, we utilize the matrix exponential function $ e^M := \sum_{n=0}^{\infty}\frac{M^n}{n!} $, $M$ square, to give
   \begin{lem}\label{lem:pascalExp}
        The relation
            \begin{eqnarray*}
                P_r(z) = e^{z A_r} 
            \end{eqnarray*}
        holds. 
    \end{lem}
    \begin{proof}
        Exercise.
    \end{proof}

    \begin{notn}
        For what is to follow, we write
            \begin{eqnarray*}
                \mathbf{G}_U(z) &:=& 
                    \begin{pmatrix}
                        g_0(z) & g_1(z) & \cdots & g_r(z)
                    \end{pmatrix}^{\mathbf{t}},\\
        \       \mathbf{F}_U(z) &:=& 
                    \begin{pmatrix}
                        f_0(z) & f_1(z) & \cdots & f_r(z)
                    \end{pmatrix}^{\mathbf{t}}
            \end{eqnarray*}
        where we have made the dependence of these vectors on the quasiautomorphic form $ U $ explicit. 
    \end{notn}

    \begin{lem}\label{lem:expForm}
        The relation 
            \begin{eqnarray*}
                \mathbf{F}_U(z) = e^{z A_r} \mathbf{G}_U(z)
            \end{eqnarray*}
        holds. 
    \end{lem}
    \begin{proof}
        Using Lemma \ref{lem:pascalExp} and \eqref{eq:matrixConvolutionIdentity} the result is immediate. 
    \end{proof}

    \begin{rmk}
        What the simple Lemma \ref{lem:expForm} makes clear is $ \mathbf{F}_U $ is a product of the \textbf{one-parameter group} $ e^{z A_r} $ (see \cite{Baker2002}, p. 58) with the the vector-valued function $ \mathbf{G}_U(z) $, which suggests some special importance of $ \mathbf{G}_U(z) $. 
    \end{rmk}

We can also give a second simple and useful description of $ \mathbf{F}_U $. Again, depending on $ U $, suppose that we write (along lines suggested by $ P_r(z) $)
    \begin{eqnarray}\label{eq:transferMatrix}
        P(\mathbf{G}_U)(z) :=
            \begin{pmatrix}
                \binom{0}{0}g_0(z) \\
                \binom{1}{0} g_1(z) & \binom{1}{1}g_0(z) \\
                \vdots & \vdots & \ddots \\
                \binom{r}{0}g_r(z) & \binom{r}{1}g_{r-1}(z) & \cdots & \binom{r}{r}g_{0}(z)  
            \end{pmatrix}  
    \end{eqnarray}
and to denote the monomial-powers vector we write 
    \begin{equation*}
        \nu_r(z) := 
        \begin{pmatrix}
            1 & z & \cdots & z^r 
        \end{pmatrix}^\mathbf{t}.
    \end{equation*}

The following fact is a simple application of this notation.

    \begin{lem}\label{lem:VectorFormEquivalence}
        The relation
            \begin{eqnarray*}
                \mathbf{F}_{U}(z) = P(\mathbf{G}_U)(z)\,\nu_r(z)
            \end{eqnarray*}
        holds. 
    \end{lem}
    \begin{proof}
        This is clear, considering the meaning of the notation  $ P(\mathbf{G}_U)(z) $ alongside \eqref{eq:matrixConvolutionIdentity}.
    \end{proof}

In a culmination of the section we introduce a
    \begin{defin}\label{def:HeckeVectorForm}
        We introduce several new terms for when discussing our theory:
            \begin{itemize}
                \item the vector-valued function $ \mathbf{F}_U\left(z \right) $ is the \textbf{Hecke vector-form} of $ U $; 
                \item the vector-valued function $ \mathbf{G}_U(z) $ of Lemma \ref{lem:expForm} is the \textbf{hauptbuch} of $U$; 
                \item the matrix function $ P(\mathbf{G}_U)(z) $ of Lemma \ref{lem:VectorFormEquivalence} is the \textbf{transfer matrix } of the hauptbuch. 
            \end{itemize}
    \end{defin}

Before leaving this section, Definition \ref{def:HeckeVectorForm} comes originally from the author's dissertation \cite{henry2025}, but here the presentation is simpler and perhaps more familiar; however, in \cite{henry2025} we prove much more about $ \mathbf{F}_U(z) $ than just the main theorem of this paper, including some results useful for the combinatorial purposes mentioned in the introduction.

\section{The proof, part I: $ U(z) \rightarrow \mathbf{F}_U(z)$}

The next goal will be to determine the behavior of Hecke vector forms under the generators $ T $ and $ S $ of $ \mathfrak{t}_\mu $. 

\subsection{Determination of $ \varepsilon_r(T) $}

We warn the reader that here we have dropped the subscript $\mathfrak{t}_\mu $ in our notation to improve legibility in many places. Facts we must first recall are, if $ H_{w-2k} \in \mathcal{A}_{w-2k}(\mathfrak{t}_\mu)$, then
    \begin{eqnarray}\label{eq:automorphFunctionalEqs}
        H_{w-2k}(Tz) = H_{w-2k}(z)\quad \text{and}\quad
        \frac{H_{w-2k}(Sz)}{z^{w-2k}} = H_{w-2k}(z) 
    \end{eqnarray}
hold; also, where $ C_{\mathfrak{t}_\mu} $ is as in \eqref{ex:defGAux}, then
    \begin{eqnarray}\label{eq:E2FunctionalEqs}
        E_2(Tz) = E_2(z) \quad \text{and}\quad
        \frac{E_2(Sz)}{z^2} = E_2(z) - {C_{\mathfrak{t}_\mu}}Sz
    \end{eqnarray}
hold \cite{doranEtAl2013}.

    \begin{lem}\label{lem:hauptbuchUnderT}
        Let $ U_{\mathfrak{t}_\mu, w, r} $ be some quasiautomorphic form. If $\mathbf{G}_U(z) $ is the hauptbuch of $ U_{\mathfrak{t}_\mu, w, r} $, then 
            \begin{eqnarray*}
                \mathbf{G}_U(Tz) = \mathbf{G}_U(z)
            \end{eqnarray*}
        holds.
    \end{lem}
    \begin{proof}
        From \eqref{eq:automorphFunctionalEqs} and \eqref{eq:E2FunctionalEqs}, we see that that $ U_{\mathfrak{t}_\mu, w, r} $ is periodic under $ Tz $ and thus $ g_\ell $ is periodic under $ Tz $.
    \end{proof}
Now it is direct that $\mathbf{F}_U $ has a Fourier series expansion, given Lemma \ref{lem:expForm} in conjunction with Lemma \ref{lem:hauptbuchUnderT}.

    \begin{cor}
        The relation
            \begin{eqnarray*}
                P(\mathbf{G}_U)(Tz) = P(\mathbf{G}_U)(z) 
            \end{eqnarray*}
        holds. 
    \end{cor}
    \begin{proof}
        This is clear from Lemma \ref{lem:hauptbuchUnderT}.
    \end{proof}

    \begin{thm}[$ \mathbf{F}_U $ under $ Tz $]\label{thm:vectorFormUnderT}
        Let $ U_{\mathfrak{t}_\mu, w,r} $ be as before. The relation 
            \begin{eqnarray*}
                \mathbf{F}_U(Tz) = \varepsilon_r(T) \, \mathbf{F}_U(z), \qquad \varepsilon_r(T) := e^{\varpi_\mu A_r}
            \end{eqnarray*}
        holds. 
    \end{thm}
    \begin{proof}
        This follows from Lemma \ref{lem:expForm}, Lemma \ref{lem:hauptbuchUnderT}, and the fact that $ e^{(a+b)M} = e^{aM}e^{bM} $ holds for scalars $ a $ and $b$. 
    \end{proof}

\subsection{Determination of $\varepsilon_r(S)$}

The case of $ \mathbf{F}_U $ under $ Sz $ is a little more involved. We will build up to the statement.  It will become clear at this stage why we have used the convolution structure. 
    \begin{prop}\label{prop:gUnderS}
        The relation 
            \begin{eqnarray*}
                \frac{g_\ell(Sz)}{z^{w-r-\ell}} = \sum_{m=0}^{r-\ell}\binom{r-\ell}{m}g_{\ell+m}z^{r-\ell - m}
            \end{eqnarray*}
        holds.
    \end{prop}
    \begin{proof}
        Starting with $ g_\ell $ as defined at \eqref{ex:defGAux}, also recalling the binomial theorem and property \eqref{eq:E2FunctionalEqs} of $ E_2 $, the string of relations
            \begin{eqnarray*}
                g_\ell(Sz) &=& C^\ell \sum_{m=0}^{r-\ell} \{r,\ell\}_m B_{U, \ell + m}(Sz)\,E_2^m(Sz) \\
                &=& C^\ell \sum_{m=0}^{r-\ell} \{r,\ell\}_m z^{w-2(\ell+m)}B_{U, \ell + m}(z)\,(z^2E_2(z) + Cz)^m \\
                &=& C^\ell \sum_{m=0}^{r-\ell} \{r,\ell\}_m z^{w-2\ell}B_{U, \ell + m}(z)\,\left(E_2(z) + \frac{C}{z} \right)^m \\
                &=& C^\ell \sum_{m=0}^{r-\ell} \{r,\ell\}_m z^{w-2\ell}B_{U, \ell + m}(z)\,\left(\sum^m_{k=0}\binom{m}{k}E_2^{m-k}(z)\left( \frac{C}{z}\right)^{k} \right)
            \end{eqnarray*}
        hold. Apply properties of sums to get something more orderly, namely
            \begin{eqnarray*}
                C^\ell \sum_{m=0}^{r-\ell} \{r,\ell\}_m z^{w-2\ell}B_{U, \ell + m}(z)\,\left(\sum^m_{k=0}\binom{m}{k}E_2^{m-k}(z)\left( \frac{C}{z}\right)^{k} \right) &=& \\ 
                C^\ell \sum_{m=0}^{r-\ell}\sum^m_{k=0} \binom{m}{k}\{r,\ell\}_m z^{w-2\ell}B_{U, \ell + m}(z)E_2^{m-k}(z) \left( \frac{C}{z}\right)^{k} &=& \\
                z^{w-2\ell} {C^\ell}\sum_{m=0}^{r-\ell}\sum^m_{k=0} \binom{m}{k}\{r,\ell\}_m B_{U, \ell + m}(z)E_2^{m-k}(z) \left( \frac{C}{z}\right)^{k} 
            \end{eqnarray*}
        hold. Apply the law 
            \begin{eqnarray*}
                \sum_{m=0}^{r-\ell} \sum_{k=0}^{m}a_{k,m} = \sum_{k=0}^{r-\ell} \sum_{m=0}^{r-\ell -k}a_{k, m+k}
            \end{eqnarray*}
        to get (after dropping arguments from notation of $ B_{U, n}(z) $ and $ E_2(z) $ and applying $ p \mapsto k $)
            \begin{eqnarray}\label{ex:penultimate}
                z^{w-2\ell}C^\ell \sum_{p=0}^{r-\ell}\sum^{r-\ell-p}_{m=0} \binomer{m+1}{p}\{r,\ell\}_{m+1} B_{U, \ell + m+p}E_2^{m}\left( \frac{C_{\mathfrak{t}_\mu}}{z}\right)^{p}
            \end{eqnarray}
        is satisfied. It is desired that we ultimately have 
            \begin{eqnarray*}
                z^{w-\ell}\sum_{p=0}^{r-\ell}\binom{r-\ell}{p} g_{\ell + p}z^{-\ell - p}
            \end{eqnarray*}
        where $ g_{\ell + p} := C^{\ell + p}\sum^{r-\ell-p}_{m=0}\{r,\ell+m\}_p B_{U,\ell+ p+m}E_2^{p} $. Using \eqref{ex:penultimate} and what we saw about properties of $ \{r,\ell\}_m $ at \eqref{eq:coeffEquiv}, we can write
            \begin{eqnarray*}
                g_{\ell}(Sz) &=& z^{w-\ell} \sum_{p=0}^{r-\ell}\sum^{r-\ell-p}_{m=0} \binom{r-\ell}{m}\{r,\ell+m\}_{p}\, B_{U, \ell +p+m}E_2^{m}\left( \frac{C}{z}\right)^{\ell+p} \\
                &=& z^{w-\ell} \sum_{p=0}^{r-\ell}\binom{r-\ell}{m}\left( {C}^{\ell+p}\sum^{r-\ell-p}_{m=0}  \{r,\ell+m\}_{p}\, B_{U, \ell +p+m}E_2^{p}\right){z}^{-\ell-p} \\
                &=& z^{w-\ell}\sum_{p=0}^{r-\ell}\binom{r-\ell}{p} g_{\ell + p}z^{-\ell - p}
            \end{eqnarray*}
        holds, where in the penultimate line, due to the invariance, we have applied $ E_2^p \mapsto E_2^m $. Therefore, we have shown that
            \begin{eqnarray*}
                \frac{g_\ell(Sz)}{z^{w-r-\ell}} = \sum^{r-\ell}_{p=0}\binom{r-\ell}{p}g_{\ell + p}z^{r-\ell-p}
            \end{eqnarray*}
        holds, which was our goal.
    \end{proof}
Proposition \ref{prop:gUnderS} is the principal tool for this article and so we have given a pedantic proof. One important implication is found in 
    \begin{cor}\label{cor:convolvUnderS}
        The relation 
            \begin{eqnarray*}
                \frac{1}{z^{w-r}}\sum_{\ell=0}^{n}\binom{n}{\ell}{g_{\ell}(Sz)}\,{(Sz)^{n-\ell}} = (-1)^n\sum_{m=0}^{r-n}\binom{r-n}{m}g_m(z)\,z^{r-n-m}
            \end{eqnarray*}
        holds. 
    \end{cor}
    \begin{proof}
        Observe that
            \begin{eqnarray*}
               \frac{1}{z^{w-r}} \sum_{\ell=0}^{n}\binom{n}{\ell}{g_{\ell}(Sz)}\,{(Sz)^{n-\ell}} &=& \frac{(Sz)^n}{z^{w-r}} \sum_{\ell=0}^{n}\binom{n}{\ell}\frac{g_\ell(Sz)}{(Sz)^\ell} \\
               &=& \frac{1}{z^{w-r}}\frac{(-1)^n}{z^n}\sum_{\ell = 0}^{n}(-1)^\ell \binom{n}{\ell}{g_\ell(Sz)}{z^\ell}
            \end{eqnarray*}
        are satisfied. With Proposition \ref{prop:gUnderS} we see
            \begin{eqnarray*}
                {g_\ell(Sz)}\,z^\ell = z^{w}\sum_{m=0}^{r-\ell}\binom{r-\ell}{m}\frac{g_{\ell + m}(z)}{z^{\ell + m}}
            \end{eqnarray*}
        holds and thus 
            \begin{eqnarray*}
               &{}& \frac{1}{z^{w-r}}\frac{(-1)^n}{z^n}\sum_{\ell = 0}^{n}(-1)^\ell \binom{n}{\ell}{g_\ell(Sz)}{z^\ell} \\ &=& \frac{(-1)^n}{z^{n-r}}\sum_{\ell = 0}^{n}(-1)^\ell \binom{n}{\ell}\left(\sum_{m=0}^{r-\ell}\binom{r-\ell}{m}\frac{g_{\ell + m}(z)}{z^{\ell + m}} \right)\\  &=& \frac{(-1)^n}{z^{n-r}}\sum_{\ell = 0}^{n}\sum_{m=0}^{r-\ell}(-1)^\ell \binom{n}{\ell}\binom{r-\ell}{r-(\ell + m)} \,\frac{g_{\ell + m}(z)}{z^{\ell + m}} 
            \end{eqnarray*}
        is correct. Applying the natural choice of sum law we get
            \begin{eqnarray*}
                &{}& \frac{(-1)^n}{z^{n-r}}\sum_{\ell = 0}^{n}\sum_{m=0}^{r-\ell}(-1)^\ell \binom{n}{\ell}\binom{r-\ell}{r-\ell - m} \,\frac{g_{\ell + m}}{z^{\ell + m}}\\ &=& \frac{(-1)^n}{z^{n-r}}\sum_{m = 0}^{r}\left(\sum_{\ell=0}^{n}(-1)^\ell \binom{n}{\ell}\binom{r-\ell}{r-m}\right) \frac{g_{m}}{z^{m}}\\ \\
                &=& {(-1)^n}\sum_{m = 0}^{r-n}\binom{r-n}{m} {g_{m}(z)}\,{z^{r-n- m}}
            \end{eqnarray*}
        holds, using the Vandermonde-like identity 
            \begin{eqnarray*}
                \sum_{\ell=0}^k (-1)^k\binom{k}{\ell}\binom{y-\ell}{p} = \binom{y-k}{y-p}
            \end{eqnarray*}
        in the penultimate line to get rid of a summation sign. 
    \end{proof}

Corollary \ref{cor:convolvUnderS} is the analogue of the transformation behavior of an automorphic form under $ S $ in $ g_\ell $. That is, were we to have written $ f_k(z) = (\mathbf{g}\oplus\mathbf{z})_k(z) $ \emph{a la} Section 2 (near \eqref{ex:fAux}), then we see  
    \begin{eqnarray*}
        \frac{(\mathbf{g}\oplus\mathbf{z})_k(Sz)}{z^{w-r}} = (-1)^k (\mathbf{g}\oplus\mathbf{z})_{r-k}(z)
    \end{eqnarray*}
holds. This fact is used in proof  of the transformation behavior of $ \mathbf{F}_U (Sz) $, but in slightly different language.

    \begin{thm}[$ \mathbf{F}_U $ under $ Sz $]\label{thm:vectorFormUnderS}
        For $ d_r $ as defined in \eqref{eq:diagMatrix}, the relation
            \begin{eqnarray}\label{eq:FUnderS}
                \frac{\mathbf{F}_U(Sz)}{z^{w-r}} = \varepsilon_r(S)\, \mathbf{F}_{U}(z), \qquad \varepsilon_r(S) = d_r^y((-1)^{r-j})  
            \end{eqnarray}
        holds.
    \end{thm}
    \begin{proof}
       For this, we use Lemma \ref{lem:VectorFormEquivalence}. That is, we will consider 
            \begin{eqnarray*}
                \mathbf{F}_U(z) = P(\mathbf{G}_U)(z)\,\nu_r(z) 
            \end{eqnarray*}
        under the map $ Sz $. Unpacking the identity gives
            \begin{eqnarray*}
                &{}& P(\mathbf{G}_U)(Sz)\,\nu_r(Sz) = \\ &{}&
                   \begin{pmatrix}
                        \binom{0}{0}g_0(Sz) \\
                        \binom{1}{0} g_1(Sz) & \binom{1}{1}g_0(Sz) \\
                        \vdots & \vdots & \ddots \\
                        \binom{r}{0}g_r(Sz) & \binom{r}{1}g_{r-1}(Sz) & \cdots & \binom{r}{r}g_{0}(Sz)  
                    \end{pmatrix}
                    \begin{pmatrix}
                        1 \\
                        Sz \\
                        \vdots\\
                        (Sz)^r
                    \end{pmatrix}
            \end{eqnarray*}
        holds. This leads us to consider only the column vector
            \begin{eqnarray*}
                \mathbf{F}_U(Sz) = 
                \begin{pmatrix}
                    g_0(Sz) \\
                    \sum_{k=0}^{1} \binom{1}{k}g_{1-k}(Sz)(Sz)^k \\
                    \vdots \\
                    \sum_{k=0}^{r} \binom{r}{k}g_{r-k}(Sz)(Sz)^k
                \end{pmatrix}
            \end{eqnarray*}
        and anticipate that we will use Corollary \ref{cor:convolvUnderS}. Indeed, applying Corollary \ref{cor:convolvUnderS} we see \eqref{eq:FUnderS} holds as we wanted. 
    \end{proof}

\section{The proof, part II: $ \mathbf{F}_U(z) \rightarrow U(z) $}
Having determined th multipliers $\varepsilon_r(T) $ and $ \varepsilon_r(S) $ in the forward direction, we will use these now to establish the bijection. In this section, finding that
    \begin{eqnarray*}
        \varpi_\mu = 2 \cos \frac{\pi}{\mu} =  e^{\frac{\pi i }{\mu}} + {e^{\frac{-\pi i }{\mu}}}
    \end{eqnarray*}
holds, we will write $ T $ in terms of the latter expression.

\begin{thm}
    Let $ \mu \geq 3 $ be an integer, let $\zeta$ denote a primitive $2\mu$-th root of unity, and let $ K =\QQ(\zeta)$. Then for all $r\geq 0$ there exists a representation $\varepsilon_r \colon \mathfrak{t}_\mu \to \GL_{r+1}(K) $    
    if
        \begin{eqnarray*}
            \varepsilon_r(u) = 
            \begin{cases}
                \exp((\zeta+\zeta^{-1})A_r) \quad & \text{if } u = T, \\
                d_r^{\text{y}}((-1)^j) & \text{if }u = S.
            \end{cases}
        \end{eqnarray*}
    holds. In fact, for $ r\geq 1 $ we have
        \begin{eqnarray*}
            \varepsilon_r \cong \Sym^{r}\varepsilon_1
        \end{eqnarray*}
    is true. 
\end{thm}
\begin{proof}
    Utilizing $ \varepsilon_r$ we will induct over $r$.  
    
    When $r=0$ this is the trivial representation. 
    
    Now set $r=1$, so that $\varepsilon_1(S)^2 = -I$ and from which we find
    \[
    \varepsilon_1(T) = \exp\!\left((\zeta+\zeta^{-1})\left(\begin{smallmatrix}
        0&0\\
        1&0
    \end{smallmatrix}\right)\right) = \left(\begin{matrix}
        1&0\\\
        \zeta+\zeta^{-1}&1
    \end{matrix}\right),
    \]
    holds. Thus,
    \[
    \varepsilon_1(R) = \varepsilon_1(S^{-1})\varepsilon_1(T) = \left(\begin{matrix}
        0&1\\
        -1&0
    \end{matrix}\right)\left(\begin{matrix}
        1&0\\\
        \zeta+\zeta^{-1}&1
    \end{matrix}\right) = \left(\begin{matrix}
        \zeta+\zeta^{-1}&1\\
        -1&0
    \end{matrix}\right)
    \]
    is satisfied. We see that the characteristic polynomial of $\varepsilon_1(R)$ is 
        \begin{eqnarray*}
            X^2-(\zeta+\zeta^{-1})X+1
        \end{eqnarray*}
    so that $\varepsilon_1(R)$ is diagonalizable with eigenvalues $\zeta$ and $\zeta^{-1}$. Since $ \zeta^\mu =\zeta^{-\mu}=-1$, we find that $\varepsilon_1(R)^\mu$ is diagonalizable with both eigenvalues equal to $-1$, and hence $\varepsilon_1(R)^\mu = -I = \varepsilon_1(S)^2$. This proves the claim when $r=1$.

    Now for $r\geq 2$ it suffices to show that there exists a basis for $\Sym^r\varepsilon_1$ such that $ S$ and $ T $ act via the matrices $\varepsilon_r(S)$ and $\varepsilon_r(T)$; the identity $\varepsilon_r(R)^\mu = \varepsilon_r(S)^2$ then follows immediately from $\Sym^r\varepsilon_1 $ being a group homomorphism and the $r=1$ case. 
    
    Recall that if we let the standard basis for $\Sym^1\varepsilon_1 = \varepsilon_1$ be $e$ and $f$ where
        \begin{eqnarray*}
            e := \left(
                \begin{matrix}
                    1\\0
                \end{matrix}\right) 
            \quad \text{and} \quad
            f := \left(
                \begin{matrix}
                    0\\1
                \end{matrix}\right),
    \end{eqnarray*}
    then the vector space underlying $\Sym^r\varepsilon_1$ is spanned by monomials $e^if^j$ with $i+j=r$, and $\gamma\in \mathfrak{t}_\mu $ acts by
        \begin{equation}\label{eq:sympower}
            \left(\Sym^r\varepsilon_1(\gamma)\right) \cdot e^if^j = (\varepsilon_1(\gamma)e)^i(\varepsilon_1(\gamma)f)^j.    
        \end{equation}
    Now consider the rescaled basis $ w_k = \frac{e^k f^{r-k}}{k!\,(r-k)!}$ for $ k=0,1,\ldots,r $.

    With respect to our alternative basis, we now describe the {action of $ S $.} Since $ \varepsilon_1(S)\cdot e = f $ and $ \varepsilon_1(S) \cdot f = -e $, equation \eqref{eq:sympower} gives
        \begin{eqnarray*}
            \left(\Sym^r\varepsilon_1(S)\right)\cdot w_k = \frac{f^k\cdot(-e)^{r-k}}{k!\,(r-k)!} = (-1)^{r-k}\,w_{r-k}.
        \end{eqnarray*}
    Hence in the basis $(w_0, w_1, \ldots, w_r)$ the matrix of $\Sym^r\varepsilon_1(S)$ is anti-diagonal with $(k,r-k)$-entry equal to $(-1)^{r-k}$. The upper-right entry (at position $(0,r)$) equals $(-1)^r$, and the signs alternate along the anti-diagonal, exactly as described for $\varepsilon_r(S)$.

    As for the action of $ T $, since $\varepsilon_1(T)\cdot e = e + (\zeta+\zeta^{-1})f$ and $\varepsilon_1(T)\cdot f = f$, equation \eqref{eq:sympower} gives
    \begin{align*}
    \left(\Sym^r\varepsilon_1(T)\right)\cdot w_k
    &= \frac{\bigl(e+(\zeta+\zeta^{-1})f\bigr)^k f^{r-k}}{k!\,(r-k)!} \\
    &= \frac{1}{k!\,(r-k)!}\sum_{m=0}^{k}\binom{k}{m}(\zeta+\zeta^{-1})^m e^{k-m}f^{r-k+m}.
    \end{align*}
    Setting $\ell = k-m$ (so the monomial is $e^\ell f^{r-\ell}$), this becomes
    \[
    \sum_{\ell=0}^{k} \frac{(\zeta+\zeta^{-1})^{k-\ell}}{(k-\ell)!\,(r-k)!}\cdot \frac{e^\ell f^{r-\ell}}{\ell!\,(r-\ell)!}\cdot \ell!\,(r-\ell)!
    = \sum_{\ell=0}^{k} \frac{\ell!\,(\zeta+\zeta^{-1})^{k-\ell}}{\ell!\,(k-\ell)!} \cdot \frac{(r-\ell)!}{(r-k)!} \cdot w_\ell,
    \]
    which simplifies to
    \[
    \left(\Sym^r\varepsilon_1(T)\right)\cdot w_k = \sum_{\ell=0}^{k} \frac{\ell!\,(\zeta+\zeta^{-1})^{k-\ell}}{(k-\ell)!}\, w_\ell.
    \]
    Thus, in the basis $(w_k)$, the $(\ell,k)$-entry of $\Sym^r\varepsilon_1(T)$ (for $\ell \leq k$) is
    \[
    \frac{\ell!\,(\zeta+\zeta^{-1})^{k-\ell}}{(k-\ell)!}.
    \]
    On the other hand, since $(A_r)_{i,j}= i\,\delta_{i,j+1}$, we compute $(A_r^m)_{i,j} = \frac{i!}{(i-m)!}\,\delta_{i,j+m}$, and hence
    \[
    \bigl[\exp\!\left((\zeta+\zeta^{-1})A_r\right)\bigr]_{\ell,k} = \frac{(\zeta+\zeta^{-1})^{\ell-k}}{(\ell-k)!}\cdot\frac{\ell!}{k!}\Bigg|_{\substack{\text{lower}\\\text{triangular}}}
    \]
    Thus, $A_r$ is sub-diagonal and so $\exp(A_r)$ is lower-triangular with $(\ell,k)$-entry (for $\ell \geq k$) equal to
    \[
    \frac{(\zeta+\zeta^{-1})^{\ell-k}}{(\ell-k)!}\cdot\frac{\ell!}{k!}.
    \]
    Renaming $(\ell, k) \mapsto (k, \ell)$ to match the $\Sym^r$ computation above, we see the two expressions agree:
    \[
    \frac{\ell!\,(\zeta+\zeta^{-1})^{k-\ell}}{(k-\ell)!\,} = \frac{(\zeta+\zeta^{-1})^{k-\ell}}{(k-\ell)!}\cdot \ell!,
    \]
    This shows that in the $w$-basis the $(\ell, k)$-entry of $\Sym^r\varepsilon_1(T)$ equals the $(\ell,k)$-entry of $\exp((\zeta+\zeta^{-1})A_r)$, completing the verification that $\varepsilon_r(T) = \Sym^r\varepsilon_1(T)$ in the $w$-basis. This completes the proof.
\end{proof}

\printbibliography

\end{document}